%% file: 3.tex
\documentclass[12pt]{article}
\usepackage{graphicx}
\usepackage{amsmath}
\usepackage{amsfonts}
\usepackage{amsthm}
\usepackage[T1]{fontenc}
\usepackage{url}
\usepackage[margin=1in]{geometry}
\input{mymacros}

\begin{document}
\title{When the Cauchy inequality becomes a formula}
\author{ Davit Harutyunyan\footnote{\'Ecole Polytechnique F\'ed\'eral de Lausanne, davit.harutyunyan@epfl.ch}}
\maketitle
\begin{abstract}
  In this note we revisit the classical geometric-arithmetic mean inequality and find a formula for the difference
 of the arithmetic and the geometric means of given $n\in\mathbb N$ nonnegative numbers $x_1,x_2,\dots,x_n$. The formula yields new stronger versions of the geometric-arithmetic mean inequality. We also find a second version of a strong geometric-arithmetic mean inequality and
 show that all inequalities are optimal in some sense. Anther striking novelty is, that the equality in all new inequalities holds not only in the case when all $n$ numbers are equal, but also in other cases.
\end{abstract}

\section{The Cauchy Inequality}

As the topic is most classical and probably the most known one among mathematicians (and could be even school students) we do not
devote a special introduction section to it, but just recall the inequality, that is due to Augustine-Louis Cauchy [\ref{bib:Cauchy}], see also
the books [\ref{bib:Arn.Arn.1},\ref{bib:Steele}] and the references therein for detailed review of the inequality.
\begin{theorem}(Cauchy or geometric-arithmetic mean inequality)
\label{th:1.1}
Assume $n\in\mathbb N$ is a natural number. Then for any nonnegative numbers $x_1,x_2,\dots,x_n$ the inequality holds
\begin{equation}
\label{1.1}
\frac{x_1+x_2+\dots +x_n}{n}\geq (x_1x_2\dots x_n)^{1/n}.
\end{equation}
\end{theorem}
Denote for convenience the arithmetic and geometric means by (by skipping the $x_i$ variable dependence)
\begin{equation}
\label{1.2}
A=\frac{x_1+x_2+\dots +x_n}{n},\qquad G=(x_1x_2\dots x_n)^{1/n}.
\end{equation}
Then the Cauchy inequality reads as
$$A-G\geq0.$$
We aim to find a formula for the difference $A-G$ that is a sum of squares, i.e., it is obviously always nonnegative. The formula is provided in the next section which is the main result of the paper.

\section{The new formula}
\setcounter{equation}{0}

In this section we prove a formula for the difference $A-G$ by iterating an inequality on $A-G.$ We prove the following theorem:
\begin{theorem}(The new formula)
\label{th:2.1}
Assume $n\in\mathbb N$ is a natural number. Then for any nonnegative numbers $x_1,x_2,\dots,x_n$ the equality holds
\begin{equation}
\label{2.1}
A-G=\sum_{k=1}^\infty2^{k-1}G^{1-\frac{1}{2^{k-1}}}\frac{1}{n}\sum_{i=1}^n\left(x_i^{\frac{1}{2^k}}-G^{\frac{1}{2^k}}\right)^2.
\end{equation}
\end{theorem}

\begin{proof}
If $G=0,$ then (\ref{2.1}) is trivial. Assume in what follows $x_i>0$ for all $i.$  Denote $y_k=2^{k-1}G^{1-\frac{1}{2^{k-1}}}\frac{1}{n}\sum_{i=1}^n(x_i^{\frac{1}{2^k}}-G^{\frac{1}{2^k}})^2$ and $Y_m=\sum_{k=1}^my_k.$ Let us show by induction in $m$ that
\begin{equation}
\label{2.2}
A-G-Y_m=2^mG^{1-\frac{1}{2^m}}\left(\frac{\sum_{i=1}^nx_i^{\frac{1}{2^m}}}{n}-G^{\frac{1}{2^m}}\right).
\end{equation}
For $m=1$ we have by simple manipulation, that
\begin{align}
\label{2.3}
A-G-Y_1&=\frac{1}{n}\sum_{i=1}^nx_n-G-\frac{1}{n}\sum_{i=1}^n(\sqrt{x_i}-\sqrt{G})^2\\ \nonumber
&=2\sqrt{G}\left(\frac{1}{n}\sum_{i=1}^n\sqrt{x_i}-\sqrt{G}\right),
\end{align}
thus the case $m=1$ is proven. Also, formula (\ref{2.3}) shows how one must calculate the difference $A-G-Y_{m+1}$ having the formula (\ref{2.2}) for $A-G-Y_m.$ Indeed, assuming that (\ref{2.2}) holds for $m$ we have, that
\begin{align}
\label{2.4}
A-G-Y_{m+1}&=A_G-Y_m-y_{m+1}\\ \nonumber
&=2^mG^{1-\frac{1}{2^m}}\left(\frac{\sum_{i=1}^nx_i^{\frac{1}{2^m}}}{n}-G^{\frac{1}{2^m}}\right)
-\frac{1}{n}2^{m}G^{1-\frac{1}{2^{m}}}\sum_{i=1}^n(x_i^{\frac{1}{2^{m+1}}}-G^{\frac{1}{2^{m+1}}})^2\\ \nonumber
&=2^{m}G^{1-\frac{1}{2^{m}}}\left(\frac{\sum_{i=1}^nx_i^{\frac{1}{2^m}}}{n}-G^{\frac{1}{2^m}}
-\sum_{i=1}^n(x_i^{\frac{1}{2^{m+1}}}-G^{\frac{1}{2^{m+1}}})^2\right).
\end{align}
Observe, that the expression in the brackets of the last line in (\ref{2.4}) is exactly of the form of $A-G-Y_1$ in (\ref{2.3}) written for the sequence $\{x_i^{\frac{1}{2^m}}\}$ with the geometric mean $G^{\frac{1}{2^m}},$ thus owing to (\ref{2.3}) we discover
$$A-G-Y_{m+1}=2^{m+1}G^{1-\frac{1}{2^{m+1}}}\left(\frac{\sum_{i=1}^nx_i^{\frac{1}{2^{m+1}}}}{n}-G^{\frac{1}{2^{m+1}}}\right).$$
the proof of (\ref{2.2}) is finished now. It remans to prove, that
\begin{equation}
\label{2.5}
\lim_{m\to\infty}2^m\left(\frac{\sum_{i=1}^nx_i^{\frac{1}{2^m}}}{n}-G^{\frac{1}{2^m}}\right)=0.
\end{equation}
We can calculate
\begin{align*}
\lim_{t\to 0}\frac{1}{t}\left(\frac{\sum_{i=1}^nx_i^{t}}{n}-G^{t}\right)&=
\frac{1}{n}\sum_{i=1}^n\lim_{t\to 0}\frac{x_i^t-1}{t}-\lim_{t\to 0}\frac{G^t-1}{t}\\
&=\frac{1}{n}\sum_{i=1}^n\ln(x_i)-\ln(G)\\
&=0,
\end{align*}
as $G$ is the geometric mean of the sequence $\{x_i\}_{i=1}^n.$ The last observation yields the validity of (\ref{2.5}). The theorem is proven now.
\end{proof}
We get the following corollary.
\begin{corollary}
\label{cor:2.2}
Assume $n\in\mathbb N$ is a natural number. Then for any nonnegative numbers $x_1,x_2,\dots,x_n$ the inequality holds
\begin{equation}
\label{2.6}
\frac{x_1+x_2+\dots +x_n}{n}-(x_1x_2\dots x_n)^{1/n}\geq\frac{1}{n}\sum_{i=1}^n(\sqrt{x_i}-\sqrt{G})^2,
\end{equation}
where $G=(x_1x_2\dots x_n)^{1/n}.$ Moreover, the coefficient $1/n$ on the right hand side of (\ref{2.6}) is optimal and the equality
in (\ref{2.6}) holds if and only if either one of the numbers $x_i$ is zero or all $x_i$ are equal.
\end{corollary}

\begin{proof}
The validity and also the case of equality of (\ref{2.6}) is trivial being a consequence of the formula (\ref{2.1}) being the obvious inequality $A-G\geq y_1.$ The optimality of the coefficient $\frac{1}{n}$ follows from the choice of the sequence $x_1=1,$ $x_i=0$ for $i=2,3,\dots,n.$
\end{proof}
Of course another inequality would be $A-G\geq Y_2,$ which reads as
\begin{equation}
\label{2.7}
A-G\geq \frac{1}{n}\sum_{i=1}^n(\sqrt{x_i}-\sqrt{G})^2
+2\sqrt{G}\frac{1}{n}\sum_{i=1}^n(x_i^{\frac{1}{4}}-G^{\frac{1}{4}})^2.
\end{equation}

\section{Another version of a stronger Cauchy inequality}
\setcounter{equation}{0}

In this section we prover following stronger version of the Cauchy inequality which has a different form than (\ref{2.6}) or (\ref{2.7}).

\begin{theorem}(Second strong Cauchy inequality)
\label{th:3.1}
Assume $n\in\mathbb N$ is a natural number. Then for any nonnegative numbers $x_1,x_2,\dots,x_n$ the inequality holds
\begin{equation}
\label{3.1}
\frac{x_1+x_2+\dots +x_n}{n}-(x_1x_2\dots x_n)^{1/n}\geq\frac{1}{n(n-1)}\sum_{1\leq i<j\leq n}(\sqrt{x_i}-\sqrt{x_j})^2
\end{equation}
and the coefficient $1/{n(n-1)}$ on the right hand side of (\ref{3.1}) is optimal.
Moreover the equality holds in (\ref{3.1}) only in one of the following cases:
\begin{itemize}
\item[(i)] If $n=2$.
\item[(ii)] All but one of the numbers $x_1,x_2,\dots,x_n$ are zero.
\item[(iii)] All of the numbers $x_1,x_2,\dots,x_n$ are equal.
\end{itemize}

\end{theorem}

\begin{proof}
the proof is done by opening the brackets on the right hand side to get the analog inequality
\begin{equation}
\label{3.2}
\sum_{1\leq i<j\leq n}\sqrt{x_ix_j}\geq\frac{n(n-1)}{2}G,
\end{equation}
which is again the geometric and arithmetic mean inequality written for the numbers $\sqrt{x_ix_j}$ for $1\leq i<j\leq n.$ Let us now analyse the equality case in (\ref{3.2}). If $n=2$ then (\ref{3.1}) obviously becomes equality. If $n>2$ then clearly the equality in (\ref{3.2}) holds if all numbers $\sqrt{x_ix_j}$ are equal for $1\leq i<j\leq n.$ If $x_i\neq 0$ for some $1\leq i\leq n,$ then we get from the equality $\sqrt{x_ix_j}=\sqrt{x_ix_k}$ that $x_j=x_k$ for $j,k\neq i.$ On the other hand as $n\geq 3$ the equality $\sqrt{x_ix_j}=\sqrt{x_jx_k}$ holds and thus we get $x_j(x_i-x_k)=0$ for all $j,k\neq i$ and $j\neq k.$ This then implies that $x_j=0$ for all $j\neq i$ or $x_i=x_j$ for all $1\leq i,j\leq n,$ which are exactly cases $(ii)$ and $(iii)$ respectively. It is trivial that both cases provide equality in (\ref{3.1}). The optimality of the constant $\frac{1}{n(n-1)}$ follows from the choice of the sequence $x_1=1,$ $x_i=0$ for $i=2,3,\dots,n.$ The proof is finished now.
\end{proof}

\section{The optimality of the exponents and constants}
 \setcounter{equation}{0}

In this section we prove the optimality of the exponent $2$ in both inequalities (\ref{2.6}) and (\ref{3.1}). It is well known that
a quantitative Cauchy inequality of the form (\ref{2.6}) yields a quantitative Brunn-Minkoski, Wulff and isoperimetric inequalities, see [\ref{bib:Ball},\ref{bib:Fig.Mag.Pra.1},\ref{bib:Fig.Mag.Pra.2},\ref{bib:Segal},] and the
the better the exponent in (\ref{2.6}) is the better the obtained Brunn-Minkowski type inequality is, thus the optimality question of $\alpha$
in both (\ref{2.6}) and (\ref{3.1}) is very important. By optimality we mean the following: find a pair of constants $C_n,\alpha>0$ such, that the inequality
\begin{equation}
\label{4.1}
A-G\geq C_n\sum_{i=1}^n{|x_i^{1/\alpha}-G^{1/\alpha}|^{\alpha}},\quad\text{for all}\quad x_1,x_2,\dots,x_n\geq 0,
\end{equation}
holds for all nonnegative sequences $x_1,x_2,\dots,x_n.$ First of all, as seen before, the choice of the sequence $x_1=1,$ $x_i=0$ for $i=2,3,\dots,n$ makes it clear, that $C_n\leq \frac{1}{n}$ in (\ref{4.1}). Assume in what follows we consider (\ref{4.1}) with $C_n=\frac{1}{n},$
i.e., the inequality
\begin{equation}
\label{4.2}
A-G\geq \frac{1}{n}\sum_{i=1}^n{|x_i^{1/\alpha}-G^{1/\alpha}|^{\alpha}},\quad\text{for all}\quad x_1,x_2,\dots,x_n\geq 0.
\end{equation}
It is straightforward to check that given $x>y>0$ numbers, the function $h(t)=(x^{1/t}-y^{1/t})^t$ decreases in the interval $(0,\infty),$ thus the optimal value of $\alpha$ in (\ref{4.1}) will be its smallest possible value. Let us prove that $\alpha\geq 2$ in (\ref{4.2}).
To that end, we test (\ref{4.2}) with the sequence $x_1=1+\epsilon,$ and $x_i=1,$ $i>1,$ where $\epsilon>0$ is a small number.
We have then from (\ref{4.2}) that
\begin{equation}
\label{4.3}
1+\frac{\epsilon}{n}-(1+\epsilon)^{1/n}\geq C_n\sum_{i=1}^n{|x_i^{1/\alpha}-G^{1/\alpha}|^{\alpha}}.
\end{equation}
The left hand side can be approximated by the binomial expansion, and we have up to the second order,
\begin{equation}
\label{4.4}
1+\frac{\epsilon}{n}-(1+\epsilon)^{1/n}=\frac{1}{2n}\left(1-\frac{1}{n}\right)\epsilon^2+O(\epsilon^3),
\end{equation}
and by the Bernoulli inequality,
\begin{align}
\label{4.5}
C_n\sum_{i=1}^n{|x_i^{1/\alpha}-G^{1/\alpha}|^{\alpha}}&\geq (n-1)C_n\left((1+\epsilon)^{1/{\alpha n}}-1\right)^\alpha\\ \nonumber
&\geq \frac{C_n}{(\alpha n)^\alpha}\epsilon^\alpha.
\end{align}
Therefore, combining (\ref{4.3})-(\ref{4.5}) we obtain for small $\epsilon,$ that
$$\frac{1}{n}\left(1-\frac{1}{n}\right)\epsilon^2\geq \frac{C_n}{(\alpha n)^\alpha}\epsilon^\alpha,$$
which then yields the estimate $\alpha\geq 2$ in the limit $\epsilon\to 0.$
Analogously is one is interested in an inequality of the form
\begin{equation}
\label{4.1}
A-G\geq C_n\sum_{1\leq i<j\leq n}^n{|x_i^{1/\alpha}-x_j^{1/\alpha}|^{\alpha}},\quad\text{for all}\quad x_1,x_2,\dots,x_n\geq 0,
\end{equation}
the the optimal values of $C_n$ and $\alpha$ are $\frac{1}{n(n-1)}$ and $2$ respectively.

\section{Acknowledgement}

The author would like to thanks Andrej Cherkaev and Elena Cherkaev for helpful advice.

\end{document}

%% file: mymacros.tex
\usepackage{amssymb,bm}
\usepackage{amsmath}
\usepackage{amsthm}
\usepackage{graphicx}
\usepackage[active]{srcltx} 
\usepackage{hyperref}
\hypersetup{pdfborder=0 0 0}

\newtheorem{theorem}{{\sc Theorem}}[section]

\newtheorem{corollary}[theorem]{Corollary}

\def\XXint#1#2#3{{\setbox0=\hbox{$#1{#2#3}{\int}$ }
\vcenter{\hbox{$#2#3$ }}\kern-.6\wd0}}

\bmdefine\BGa{\alpha}
\bmdefine\BGb{\beta}
\bmdefine\BGd{\delta}
\bmdefine\BGe{\epsilon}
\bmdefine\BGve{\varepsilon}
\bmdefine\BGf{\phi}
\bmdefine\BGvf{\varphi}
\bmdefine\BGg{\gamma}
\bmdefine\BGc{\chi}
\bmdefine\BGi{\iota}
\bmdefine\BGk{\kappa}
\bmdefine\BGl{\lambda}
\bmdefine\BGn{\eta}
\bmdefine\BGm{\mu}
\bmdefine\BGv{\nu}
\bmdefine\BGp{\pi}
\bmdefine\BGth{\theta}
\bmdefine\BGvth{\vartheta}
\bmdefine\BGr{\rho}
\bmdefine\BGvr{\varrho}
\bmdefine\BGs{\sigma}
\bmdefine\BGvs{\varsigma}
\bmdefine\BGt{\tau}
\bmdefine\BGj{\tau}
\bmdefine\BGu{\upsilon}
\bmdefine\BGo{\omega}
\bmdefine\BGx{\xi}
\bmdefine\BGy{\psi}
\bmdefine\BGz{\zeta}
\bmdefine\BGD{\Delta}
\bmdefine\BGF{\Phi}
\bmdefine\BGG{\Gamma}
\bmdefine\BGL{\Lambda}
\bmdefine\BGP{\Pi}
\bmdefine\BGT{\Theta}
\bmdefine\BGS{\Sigma}
\bmdefine\BGU{\Upsilon}
\bmdefine\BGO{\Omega}
\bmdefine\BGX{\Xi}
\bmdefine\BGY{\Psi}

\bmdefine\BCA{{\mathcal A}}
\bmdefine\BCB{{\mathcal B}}
\bmdefine\BCC{{\mathcal C}}
\bmdefine\BCD{{\mathcal D}}
\bmdefine\BCE{{\mathcal E}}
\bmdefine\BCF{{\mathcal F}}
\bmdefine\BCG{{\mathcal G}}
\bmdefine\BCH{{\mathcal H}}
\bmdefine\BCI{{\mathcal I}}
\bmdefine\BCJ{{\mathcal J}}
\bmdefine\BCK{{\mathcal K}}
\bmdefine\BCL{{\mathcal L}}
\bmdefine\BCM{{\mathcal M}}
\bmdefine\BCN{{\mathcal N}}
\bmdefine\BCO{{\mathcal O}}
\bmdefine\BCP{{\mathcal P}}
\bmdefine\BCQ{{\mathcal Q}}
\bmdefine\BCR{{\mathcal R}}
\bmdefine\BCS{{\mathcal S}}
\bmdefine\BCT{{\mathcal T}}
\bmdefine\BCU{{\mathcal U}}
\bmdefine\BCV{{\mathcal V}}
\bmdefine\BCW{{\mathcal W}}
\bmdefine\BCX{{\mathcal X}}
\bmdefine\BCY{{\mathcal Y}}
\bmdefine\BCZ{{\mathcal Z}}

\bmdefine\Bzr{ 0}
\bmdefine\Ba{ a}
\bmdefine\Bb{ b}
\bmdefine\Bc{ c}
\bmdefine\Bd{ d}
\bmdefine\Be{ e}
\bmdefine\Bf{ f}
\bmdefine\Bg{ g}
\bmdefine\Bh{ h}
\bmdefine\Bi{ i}
\bmdefine\Bj{ j}
\bmdefine\Bk{ k}
\bmdefine\Bl{ l}
\bmdefine\Bm{ m}
\bmdefine\Bn{ n}
\bmdefine\Bo{ o}
\bmdefine\Bp{ p}
\bmdefine\Bq{ q}
\bmdefine\Br{ r}
\bmdefine\Bs{ s}
\bmdefine\Bt{ t}
\bmdefine\Bu{ u}
\bmdefine\Bv{ v}
\bmdefine\Bw{ w}
\bmdefine\Bx{ x}
\bmdefine\By{ y}
\bmdefine\Bz{ z}
\bmdefine\BA{ A}
\bmdefine\BB{ B}
\bmdefine\BC{ C}
\bmdefine\BD{ D}
\bmdefine\BE{ E}
\bmdefine\BF{ F}
\bmdefine\BG{ G}
\bmdefine\BH{ H}
\bmdefine\BI{ I}
\bmdefine\BJ{ J}
\bmdefine\BK{ K}
\bmdefine\BL{ L}
\bmdefine\BM{ M}
\bmdefine\BN{ N}
\bmdefine\BO{ O}
\bmdefine\BP{ P}
\bmdefine\BQ{ Q}
\bmdefine\BR{ R}
\bmdefine\BS{ S}
\bmdefine\BT{ T}
\bmdefine\BU{ U}
\bmdefine\BV{ V}
\bmdefine\BW{ W}
\bmdefine\BX{ X}
\bmdefine\BY{ Y}
\bmdefine\BZ{ Z}

